\documentclass[12pt,a4paper]{amsart}
\usepackage{latexsym}
\usepackage{amsmath,amssymb,exercise,enumerate}
\usepackage{dsfont}
\usepackage{wasysym}

\usepackage[bookmarks=true,hyperindex,pdftex,colorlinks,citecolor=blue]{hyperref}

 \newtheorem{theo}{Theorem}[section]
 \newtheorem{definition}[theo]{Definition}
 \newtheorem{example}[theo]{Example}
  \newtheorem{lem}[theo]{Lemma}
\newtheorem{propos}[theo]{Proposition}

\newcommand{\R}{{\mathbb{R}}}

\newcommand{\N}{{\mathbb{N}}}

\newcommand{\Z}{{\mathbb{Z}}}

\newcommand{\loglike}[1]{\mathop{\rm #1}\nolimits}

\newcommand{\spn}{\loglike{Lin}}
\newcommand{\cspn}{\overline{\loglike{Lin}}}

\newcommand{\vertiii}[1]{{\left\vert\kern-0.25ex\left\vert\kern-0.25ex\left\vert #1
    \right\vert\kern-0.25ex\right\vert\kern-0.25ex\right\vert}}
    \newcommand{\vertiiis}[1]{{\vert\kern-0.25ex\vert\kern-0.25ex\vert #1
    \vert\kern-0.25ex\vert\kern-0.25ex\vert}}
\newcommand{\scal}[1]{{\left\langle\kern-0.25ex\left\langle #1
    \right\rangle\kern-0.25ex\right\rangle}}

\newcommand{\bea}{\begin{eqnarray*}}
\newcommand{\eea}{\end{eqnarray*}}
\newcommand{\beq}{\begin{eqnarray}}
\newcommand{\eeq}{\end{eqnarray}}

 \renewcommand{\le}{\leqslant}
\renewcommand{\leq}{\leqslant}
\renewcommand{\ge}{\geqslant}
\renewcommand{\geq}{\geqslant}

\numberwithin{equation}{section}
\begin{document}
\title[Plasticity of ellipsoids in Hilbert spaces]{Linear expand-contract plasticity of ellipsoids in separable Hilbert spaces}

\author{Olesia Zavarzina}
\address{School of Mathematics and Informatics, V.N. Karazin Kharkiv National University, 61022 Kharkiv, Ukraine
\newline \href{http://orcid.org/0000-0002-5731-6343}{ORCID: \texttt{0000-0002-5731-6343}}}
\email{olesia.zavarzina@yahoo.com}

\subjclass[2010]{46B20, 54E15}

\keywords{non-expansive map; ellipsoid; linearly expand-contract plastic space}
\thanks{ The research is done in frames of Ukrainian Ministry of Science and Education Research Program 0118U002036, and is partially supported by a grant from Akhiezer's Fund, 2018. }

\begin{abstract}
 The paper is aimed to establish the interdependence between linear expand-contract plasticity of an ellipsoid in a separable Hilbert spaces and properties of the set of its semi-axes.
\end{abstract}

\maketitle

\section{Introduction}
Let $M$ be a metric space and $F\colon M\to M$ be a map. $F$ is called \emph{non-expansive} if it does not increase distance between points of the space $M$. M is called \emph{expand-contract plastic} (EC-plastic for short) if every non-expansive bijection $F\colon M \to M$ is an isometry. This property of $M$ can be reformulated in the following way: for every  bijection $F \colon M\to M$, if there are points $x_1, x_2 \in M$ such that $\rho(F(x_1), F(x_2)) < \rho(x_1, x_2)$, then there are other points $y_1, y_2 \in M$ such that $\rho(F(y_1), F(y_2)) > \rho(y_1, y_2)$. It is well known that every compact metric space is EC-plastic, moreover every precompact space (in particular every bounded metric subset of a finite-dimensional normed space, equipped with the induced metric) has the same property (\cite[Satz IV]{FH1936}, see also  \cite[Theorem 1.1]{NaiPioWing}). Using bases of uniformities one can define non-expansive maps in uniform spaces and extend the above result to compact or totally bounded uniform spaces \cite[Theorem 2.3 and Corollary 2.5]{AnKaZa}.

For solid bounded subsets of infinite-dimensional normed space the situation with EC-plasticity is more complicated, because some of them do have this property and some do not, and the classification is not known even for solid bounded closed convex subsets. Even more, the intriguing question whether the unit ball of every Banach space is EC-plastic remains unsolved in spite of a number of recent papers \cite{CKOW2016}, \cite{KZ},  \cite{KZ2017}, \cite{AnKaZa} devoted to this problem.

This paper is motivated by two results by Cascales,  Kadets, Orihuela and Wingler  from \cite{CKOW2016}. Theorem 2.6 states that the unit ball of every strictly convex Banach space is EC-plastic. This implies in particular the EC-plasticity of the unit ball of Hilbert space, which is the simplest possible infinite-dimensional ellipsoid. On the other hand, there is a little bit more complicated ellipsoid which is not EC-plastic:
\begin{example}[{\cite[Example 2.7]{CKOW2016}}]
Consider $H = \ell_2(\mathbb Z)$ and
$$
A =\left\{x = (x_n) \in H: \sum_{k=-\infty}^0 |x_k|^2 + \sum_{k=1}^\infty \left|2 x_k\right|^2 \le 1 \right\}.
$$
Define the linear weighted shift operator $T \colon H \to H$ as follows: $Te_n = e_{n+1}$ for $n \neq 0$ and $Te_0 = \frac{1}{2}e_1$.
This operator maps $A$ to $A$ bijectively, is non-expansive and is not an isometry.
\end{example}
Remark that ellipsoids in finite-dimensional spaces are plastic due to compactness argument. Observe also that operator $T$ in the previous example is linear, so in infinite-dimensional spaces even linear non-expansive bijection of an ellipsoid may not be an isometry. That is why it is reasonable to introduce the following definition:
\begin{definition}
 \emph{Let $M$ be a subset of a normed space $X$. We say that $M$ is} linearly expand-contract plastic \emph{(briefly an LEC-plastic) if every linear operator $T: X \to X$ whose restriction on $M$ is a non-expasive bijection from $M$ onto $M$ is an isometry on $M$}.
 \end{definition}

The aim of this short note is to find the necessary and sufficient conditions on semi-axes of a solid bounded ellipsoid in a Hilbert space for linear plasticity of that  ellipsoid.

\section {Main result}

Below we follow the notations from \cite{Kad}. The letter $H$ denotes a fixed separable infinite-dimensional Hilbert space (real or complex), the symbol $\left\langle x, y \right\rangle$ stays for the scalar product of elements $x, y \in H$ and $\{e_i\}_{i\in \N}$ denotes a fixed orthonormal basis in $H$. Any $x \in H$ admits representation $x = \sum_{n\in \N}x_n e_n$ where $x_n$ are the corresponding Fourier coefficients. Symbol $\spn$ we use to denote the linear span, and $\cspn$ stays for the closed linear span. Ellipsoids we are going to consider generalize straightforwardly finite-dimensional ones. Namely, an ellipsoid in $H$ is a set of the form
 $$
  E =\left\{x = \sum_{n \in \N}x_n e_n  \in H: \sum_{n \in \N}\left|\frac{x_n}{a(n)}\right|^2 \le 1 \right\},
 $$
and the positive numbers $a(n) > 0$ are called \emph{semiaxes} of $E$.  Also in what follows we suppose that the corresponding function $a \colon \N \to \R^+$ is bounded above and below, that is $\inf_n a(n) > 0$ and $\sup_n a(n) < +\infty$. Denote $A = a(\N)$ the set of semi-axes of $E$. Some of semiaxes may have the same length, so we need a bit more terminology. For every $t \in A$ we call its \emph{mutiplicity} the number of elements in the set $a^{-1}(t)$ (which may be finite or infinite) and denote $H_t = \spn\{e_k\}_{k \in a^{-1}(t)}$. We will also use the notation
$$
S = \left\{x = \sum_{n \in \N}x_n e_n  \in H: \sum_{n \in \N}\left|\frac{x_n}{a(n)}\right|^2 = 1 \right\}
$$
for the boundary of $E$.

 Let us begin with not LEC-plastic ellipsoids. One may see that the example from the introduction admits generalization.
\begin{propos}\label{main_ex}
Let the set $A$ of semiaxes of the ellipsoid $E$ contain a subset $B$ possessing the following properties:
\begin{enumerate}
\item $B$ has at least two elements;
\item either $B$ doesn't have minimum or multiplicity of the minimum is infinite;
\item  either $B$ doesn't have maximum or multiplicity of the maximum is infinite.
\end{enumerate}
Then $E$ is not LEC-plastic.
\end{propos}
\begin{proof}
Denote $r = \inf B$, $R = \sup B$; according to (1) $r < R$. The property (2) ensures the existence of distinct $n_k \in \N$, $k = 1, 2, \ldots$ such that $a(n_k) \in B$, $a(n_k) < \frac12 (r + R)$ and
$$
a(n_1) \ge a(n_2) \ge a(n_3) \ge \ldots, \quad \lim_{k \to \infty} a(n_k)  = r.
$$
Analogously, the property (3) gives us the existence of distinct $n_k \in \N$, $k = 0, -1, -2, \ldots$ such that $a(n_k) \in B$ and
$$
a(n_1) < a(n_0) \le a(n_{-1}) \le a(n_{-2}) \le \ldots, \quad \lim_{k \to - \infty} a(n_k)  = R.
$$
Define the linear operator $T$ as follows: $Te_n = e_n$ for $n  \in \N \setminus \{n_k\}_{k \in \Z}$, and $Te_{n_k} = \frac{a(n_{k+1})}{a(n_k)}e_{n_{k+1}}$, for $k \in \Z$. Linear non-expansive operator $T$ maps $E$ onto itself bijectively but not isometrically.
\end{proof}

Further we are going to consider LEC-plastic ellipsoids in $H$ and prove that the negation of the conditions in the Proposition above are not only necessary, but also sufficient for linear expand-contract plasticity. In our exposition we will use the following lemmas.

\begin{lem}\label{lem_ind_0}
Let $T\colon H \to H$ be a linear operator which maps $E$ bijectively onto itself. Then $T$ maps the whole $H$ bijectively onto itself and  $T(S) = S$. If, moreover, $T$ is non-expansive on $E$, then $\|T\| \le 1$.
\end{lem}
\begin{proof}
At first, $E$ is an absorbing set, so $H = \cup_{t > 0}tE$. By linearity $T$ is injective on every set $tE$, consequently it is injective on the whole $H$. Also, $T(H) = \cup_{t > 0}T(tE) = \cup_{t > 0}tE = H$ which gives the surjectivity on $H$. Finally, $S = E \setminus \cup_{t \in (0, 1)}tE$, so $T(S) = T(E) \setminus \cup_{t \in (0, 1)}T(tE) = E \setminus \cup_{t \in (0, 1)}tE = S$.  If, moreover, $T$ is non-expansive on $E$, then for every $x \in H$ there is a $t > 0$ such that $tx \in E$ and we have $\|T(tx)\| = \rho(T(0), T(tx)) \le \rho(0, tx) = \|tx\|$. It remains to divide by $t$ and obtain that $\|Tx\| \le \|x\|$ for all  $x \in H$.
\end{proof}

\begin{lem}\label{lem_ind}
Let the set $A$ of semiaxes of the ellipsoid $E$ contain the minimal element $r$  and let $r$ have finite multiplicity. Let $T\colon H \to H$ be a linear operator  which maps $E$ bijectively onto itself and whose restriction on $H$ is non-expansive. Then $T(H_r) = H_r$, $T\left(H_r \cap E \right) = H_r \cap E$ and the restriction of $T$ onto $H_r$ is a bijective isometry.
\end{lem}
\begin{proof}

Since $T$ is non-expansive, all the elements of $S$ of minimal  norm $r$ may be mapped only to those elements of $S$ whose norm is equal to $r$. In other words, $T(H_r) \subset H_r$. For the finite-dimensional  linear space $H_r$ the injectivity of the linear map $T|_{H_r} \colon H_r \to H_r$ implies bijectivity, so $T(H_r) = H_r$ and $T\left(H_r \cap E \right) = H_r \cap E$.
Remark, that $H_r \cap E$ is equal to the closed ball of radius $r$ centered at 0 of the subspace  $H_r$. By linearity this means that $T|_{H_r}$ maps bijectively the unit ball onto the unit ball, so $T|_{H_r}$ is a bijective isometry.
\end{proof}

\begin{lem}\label{lem_ind2}
 Let the set $A$ of semiaxes of the ellipsoid $E$ contain the maximal element $R$  and let $R$ have finite multiplicity. Let $T\colon H \to H$ be a linear operator  which maps $E$ bijectively onto itself and whose restriction on $H$ is non-expansive. Then $T(H_R) = H_R$, $T\left(H_R \cap E \right) = H_R \cap E$ and the restriction of $T$ onto $H_R$ is a bijective isometry.
\end{lem}
\begin{proof}
The statement is similar to the previous one, and the proof will be similar as well. Since $T$ is non-expansive, the preimages of all the elements of $S$ of maximal  norm $R$ may be only those elements of $S$ whose norm is equal to $R$. In other words, $T^{-1}(H_R) \subset H_R$. For the finite-dimensional  linear space $H_R$ the injectivity of the linear map $T^{-1}|_{H_R} \colon H_R \to H_R$ implies bijectivity, so $T^{-1}(H_R) = H_R$ and $T(H_R) = H_R$. The rest of the proof is the same as in the previous Lemma.
\end{proof}

Now we are ready for the main theorem of the paper.

\begin{theo}
The ellipsoid $E$ is LEC-plastic iff every subset $B$ of the set $A$ of semi-axes of $E$ that consists of more than one element possesses at least one of the following properties:
\begin{enumerate}
\item $B$ has a maximum of finite multiplicity;
\item  $B$ has a minimum of finite multiplicity.
\end{enumerate}
\end{theo}
 \begin{proof}
  We have already demonstrated in Proposition \ref{main_ex} the ``only if'' part of the theorem. It remains for us to prove the ``if'' part.

Let us first note that $A$ cannot contain more than one element of infinite multiplicity. Indeed, if $b_1, b_2 \in A$ are two different elements of infinite multiplicity, then $B = \{b_1, b_2\} \subset A$ satisfies neither condition (1) nor condition (2) of our theorem.

\textbf{Claim 1}. \emph{There is a $\tau > 0$ such that $A^+ = A \cap (\tau, +\infty)$  is well-ordered with respect to the ordering $\geq$ (that is every not empty subset of $A^+$ has a maximal element), $A^- = A \cap (0, \tau)$  is well-ordered with respect to the ordering $\leq$ (that is every not empty subset of $A^-$ has a minimal element), and neither $A^+$ nor $A^-$ contain elements of infinite multiplicity}.

Indeed, if there is an element $a_\infty \in A$ of infinite multiplicity, let us take $\tau = a_\infty$. Let us demonstrate that $(A^+,\geq)$ is well-ordered. If  $A^+ = \emptyset$ the statement is clear. In the other case for every not empty subset $D$ of $A^+$ consider $B = \{\tau\} \cup D$. Then  the minimal element of $B$ is $\tau$, which has infinite multiplicity so $B$ must have a maximum of finite multiplicity. This maximum will be also the maximal element of $D$. The demonstration of well ordering for $(A^-,\leq)$ works the same way.

Now, consider the remaining case of $A$ consisting only of finite multiplicity elements. Consider the set $U$ of all those $t \in (0, +\infty)$ that $A \cap (t, +\infty)$ is not empty and well-ordered with respect to the ordering $\geq$. If $U$ is not empty, take $\tau = \inf U$, if $U = \emptyset$, take $\tau = \sup A$. Let us demonstrate that this $\tau$ is what we need. In the first case $A^+ = A \cap (\tau, +\infty)$ and for every $t > \tau$ we have $A \cap (t, +\infty)$ is not empty and well-ordered with respect to the ordering $\geq$. This implies that $(A^+,\geq)$ is well-ordered. In the second case $A^+ = \emptyset$, which is also well-ordered \smiley.
So, it remains to demonstrate that $A^- = A \cap (0, \tau)$  is well-ordered with respect to the ordering $\leq$. Assume this is not true. Then, there is a not empty subset $B \subset A^-$ with no minimal element. According to the conditions of our theorem $B$ has a maximal element $b$. Since $b < \tau$ and by definition of $\tau$ the set $A \cap (b, +\infty)$ is not well-ordered with respect to the ordering $\geq$. Consequently, there is a not empty $D \subset A \cap (b, +\infty)$ with no maximal element. Then, $B \cup D$ satisfies neither condition (1) nor condition (2) of our theorem. This contradiction completes the demonstration of Claim 1.

Now we introduce in the natural way the following three subspaces:
$$
H^- = \cspn\{e_k\}_{k \in a^{-1}(A^-)}, \,  H_\tau = \cspn\{e_k\}_{k \in a^{-1}(\tau)}, \, H^+ = \cspn\{e_k\}_{k \in a^{-1}(A^+)}.
$$
Evidently, these closed linear subspaces of $H$ are mutually orthogonal and $H = H^- \oplus H_\tau \oplus H^+$ (some of the summands may be trivial). Let $T\colon H \to H$ be a linear operator which maps $E$ bijectively onto itself and whose restriction on $E$ is non-expansive.

\textbf{Claim 2}. \emph{$T(H^-) = H^-$, $T(H^+) = H^+$ and the restrictions of $T$ onto $H^-$ and $H^+$ are bijective isometries}.

We will demonstrate the part of our claim that speaks about $H^+$: the reasoning about $H^-$ will differ only in the usage of Lemma \ref{lem_ind} instead of  Lemma \ref{lem_ind2}.

If $A^+ = \emptyset$ there is noting to do. In the case of $A^+ \neq \emptyset$
we are going to demonstrate by transfinite induction in $t \in (A^+,\geq)$ the validity for all $t \in A^+$ of the following statement $\frak U(t)$: the subspace $H(t) := \cspn\{e_k \colon a(k) \ge t\}$ is $T$-invariant and $T$ maps $H(t)$ onto $H(t)$ isometrically. Since the collection of subspaces $H(t)$, $t \in A^+$ is a chain whose union is dense in $H^+$, the continuity of $T$ will imply the desired Claim 2.

The base of induction is the statement  $\frak U(t)$ for $t = \max A$. This is just the statement of Lemma \ref{lem_ind2}. We assume now as inductive hypothesis the validity of $\frak U(t)$ for all $t > t_0 \in A^+$, and our goal is to prove the statement $\frak U(t_0)$.

For every $x,y$ in $H$ let us introduce a modified scalar product $\scal{x,y}$ as follows:
$$
\scal{x,y}=\sum_{n \in \N}\frac{x_n \overline{y_n}}{a(n)^2}.
$$
Then the norm on $H$ induced by this modified scalar product is
$$
\vertiiis{x}=\left(\sum_{i\in I}\left|\frac{x_i}{a_i}\right|^2\right)^{1/2}.
$$
Ellipsoid $E$ is the unit ball in this new norm and since $T$ is linear and maps $E$ onto $E$ bijectively, $T$ is a bijective isometry of $\left(H,  \vertiiis{\cdot}\right)$ onto itself.

Due to \cite[Theorem 2, p. 353]{Kad} $T$ is a unitary operator in the modified scalar product and thus $T$ preserves the modified scalar product. In particular, it preserves the orthogonality in the modified scalar product.

Denote $X = \bigcup_{t > t_0}H(t) = \spn\{e_k \colon a(k) > t_0\}$. The orthogonal complement to $X$ in the modified scalar product is $X^\perp =  \cspn\{e_k \colon a(k) \le t_0\}$ (occasionally the orthogonal complement to $X$ in the original scalar product is the same).  Our  inductive hypothesis implies that $T(X) = X$, consequently $T(X^\perp) = X^\perp$ and  $T(X^\perp \cap E) = X^\perp \cap E$.

$X^\perp$ equipped with the original scalar product is a Hilbert space, $X^\perp \cap E$ is an ellipsoid in $X^\perp$, $t_0$ is the maximal semiaxis of that ellipsoid  and the multiplicity of  $t_0$ is finite because  $t_0 \in A^+$. The application of Lemma \ref{lem_ind2} gives us that $T(H_{t_0}) = H_{t_0}$  and the restriction of $T$ onto $H_{t_0}$ is a bijective isometry in the original norm. Now, $T$ maps $X$ onto $X$ isometrically, maps $H_{t_0}$  onto $H_{t_0}$ isometrically and $H(t_0)$ is the orthogonal direct sum of subspaces $H_{t_0}$ and the closure of $X$. This implies that $T$ maps $H(t_0)$ onto $H(t_0)$ isometrically, and the inductive step is done. This completes the demonstration of Claim 2.

From Claim 2 and mutual orthogonality of $H^-$ and $H^+$ we deduce that  $T(H^- \oplus H^+) = H^- \oplus H^+$ and $T$ is an isometry on  $H^- \oplus H^+$. Recalling again that $T$ preserves the modified scalar product and the fact that the orthogonal complement to $X$ in the modified scalar product is $H_\tau$ we obtain that $T(H_\tau) = H_\tau$  and consequently $T(H_\tau \cap E) = H_\tau \cap E$. But $H_\tau \cap E$ is equal to the closed ball of radius $\tau$ (in the original norm) centered at 0, so the equality $T(H_\tau \cap E) = H_\tau \cap E$ and linearity of $T$ implies that $T$ is an isometry on $H_\tau$. Finally, as we know, $H = H^- \oplus H_\tau \oplus H^+$, so $T$ is an isometry on the whole $H$.
\end{proof}

\textbf{Remark}.We are primarily interested in classification of closed bounded solid sets. That is why we input the restrictions  $\inf_n a(n) > 0$ and $\sup_n a(n) < +\infty$ on the semiaxes of the ellipsoid. Nevertheless, the question of LEC-plasticity makes sense for ellipsoids with arbitrary semiaxes and the same description remains valid. The only difference is that in that general case $E$ may be not absorbing and the condition on $T$ of being non-expansive on $E$ implies continuity of $T$ on $\spn E$ but does not imply the continuity on the whole $H$. That is why in the general case all the lemmas and claims should deal with the linear span of the ellipsoid $E$ instead of the whole space $H$. In this case the subspace $\spn E$ may be non-closed and the corresponding inner product spaces $(\spn E, \langle \cdot, \cdot \rangle)$ and $(\spn E, \scal{\cdot, \cdot})$ may be incomplete, which leads to some difficulties because in most books orthonormal bases and unitary operators are defined only for Hilbertian, i.e. complete spaces. Fortunately, these difficulties are of purely terminological nature.

\vspace{3mm}
\textbf{Acknowledgement}. The author is grateful to her scientific adviser Vladimir
Kadets for constant help with this project.

\bibliographystyle{amsplain}

\end{document}